\documentclass[12pt,twoside]{article}

\usepackage{fancyhdr}
\usepackage{amsmath}
\usepackage{amssymb}
\usepackage{amsfonts}
\usepackage{euscript}
\usepackage{oldgerm}
\usepackage{calligra}
\usepackage{mathrsfs}
\usepackage{hyperref}
\usepackage{url}
\usepackage{path}
\usepackage{ecltree}
\usepackage{epsfig}
\usepackage{lastpage}
\usepackage{epic}

\usepackage{multirow}
\usepackage{graphicx}
\usepackage{stmaryrd}
\usepackage{wasysym}
\usepackage{pifont}

\usepackage{draftwatermark}
\SetWatermarkAngle{45}
\SetWatermarkLightness{0.9}
\SetWatermarkFontSize{5cm}
\SetWatermarkScale{2}
\SetWatermarkText{ }

\renewcommand{\thefootnote}{\fnsymbol{footnote}}

\long\def\sfootnote[#1]#2{\begingroup
\def\thefootnote{\fnsymbol{footnote}}\footnote[#1]{#2}\endgroup}

\newtheorem{theorem}{Theorem}[section]
\newtheorem{definition}[theorem]{Definition}

\newtheorem{lemma}[theorem]{Lemma}

\newtheorem{example}[theorem]{Example}

\newtheorem{remark}[theorem]{Remark}

\newenvironment{proof}{\noindent\mbox{\bf Proof.}}
{\hfill\mbox{\ding{113}}\bigskip}

\begin{document}
\pagestyle{fancy}
\lhead[page \thepage \ (of \pageref{LastPage})]{}
\chead[{\bf Herbrand Consistency of  Finite Fragments of ${\rm I\Delta_0}$}]{{\bf Herbrand Consistency of  Finite Fragments of ${\rm I\Delta_0}$}}
\rhead[]{page \thepage \ (of \pageref{LastPage})}
\lfoot[\copyright\ {\sf Saeed Salehi 2011}]{$\varoint^{\Sigma\alpha\epsilon\epsilon\partial}_{\Sigma\alpha\ell\epsilon\hslash\imath}\centerdot${\footnotesize {\rm ir}}}
\cfoot[{\footnotesize {\tt http:\!/\!/saeedsalehi.ir/}}]{{\footnotesize {\tt http:\!/\!/saeedsalehi.ir/}}}
\rfoot[$\varoint^{\Sigma\alpha\epsilon\epsilon\partial}_{\Sigma\alpha\ell\epsilon\hslash\imath}\centerdot${\footnotesize {\rm ir}}]{\copyright\ {\sf Saeed Salehi 2011}}
\renewcommand{\headrulewidth}{1pt}
\renewcommand{\footrulewidth}{1pt}
\thispagestyle{empty}

\begin{center}
\begin{table}
\begin{tabular}{| c | l  || l | c |}
\hline
 \multirow{7}{*}{\includegraphics[scale=0.75]{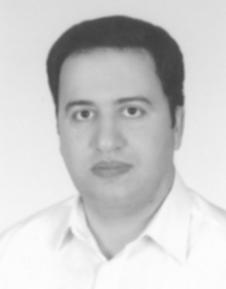}}&    &  &
 \multirow{7}{*}{ \ \ ${\huge \varoint^{\Sigma\alpha\epsilon\epsilon\partial}_{\Sigma\alpha\ell\epsilon\hslash\imath}\centerdot}${{\rm ir}}} \ \ \ \ \\
 &     \ \ {\large{\sc Saeed Salehi}}  \ \  \ & \ \    Tel: \, +98 (0)411 339 2905  \ \  &  \\
 &   \ \ Department of Mathematics \ \  \ & \ \ Fax: \ +98 (0)411 334 2102 \ \  & \\
 &   \ \ University of Tabriz \ \ \  & \ \ E-mail: \!\!{\tt /root}{\sf @}{\tt SaeedSalehi.ir/} \ \  &  \\
 &  \ \ P.O.Box 51666--17766 \ \ \ &   \ \ \ \ {\tt /SalehiPour}{\sf @}{\tt TabrizU.ac.ir/} \ \  &  \\
 &   \ \ Tabriz, Iran \ \ \ & \ \ Web: \  \ {\tt http:\!/\!/SaeedSalehi.ir/} \ \ &  \\
 &    &  &  \\
 \hline
\end{tabular}
\end{table}
\end{center}


\vspace{2em}

\begin{center}
{\bf {\Large Herbrand Consistency of  Some Finite Fragments of

\medskip

Bounded Arithmetical Theories
}}
\end{center}

\vspace{2em}
\begin{abstract}
We formalize the notion of Herbrand Consistency in an appropriate way for bounded arithmetics, and show the existence of a finite fragment of ${\rm I\Delta_0}$ whose Herbrand Consistency is not  provable in the thoery ${\rm I\Delta_0}$. We also show the existence of an ${\rm I\Delta_0}-$derivable $\Pi_1-$sentence  such that ${\rm I\Delta_0}$ cannot prove its Herbrand Consistency.

\bigskip

\bigskip

{\footnotesize
\noindent {\bf Acknowledgements} \
This research is partially supported by grant  No 89030062 of the Institute for Research in Fundamental Sciences (IPM), Tehran, Iran.
}

\centerline{${\backsim\!\backsim\!\backsim\!\backsim\!\backsim\!\backsim\!\backsim\!
\backsim\!\backsim\!\backsim\!\backsim\!\backsim\!\backsim\!\backsim\!
\backsim\!\backsim\!\backsim\!\backsim\!\backsim\!\backsim\!\backsim\!
\backsim\!\backsim\!\backsim\!\backsim\!\backsim\!\backsim\!\backsim\!
\backsim\!\backsim\!\backsim\!\backsim\!\backsim\!\backsim\!\backsim\!
\backsim\!\backsim\!\backsim\!\backsim\!\backsim\!\backsim\!\backsim\!
\backsim\!\backsim\!\backsim\!\backsim\!\backsim\!\backsim\!\backsim\!
\backsim\!\backsim\!\backsim\!\backsim\!\backsim\!\backsim\!\backsim\!
\backsim\!\backsim\!\backsim\!\backsim\!\backsim\!\backsim\!\backsim\!
\backsim\!\backsim\!\backsim\!\backsim\!\backsim\!\backsim\!\backsim\!
\backsim\!\backsim\!\backsim}$}

\bigskip



\end{abstract}

\bigskip

\bigskip

\bigskip

\bigskip

\bigskip

\bigskip

\bigskip

\bigskip

\bigskip

\vfill

\hspace{.75em} \fbox{\textsl{\footnotesize Date: 09 October 2011 (09.10.11)}}

\vfill

\bigskip
\noindent\underline{\centerline{}}
\centerline{page 1 (of \pageref{LastPage})}


\newpage
\setcounter{page}{2}
\SetWatermarkAngle{65}
\SetWatermarkLightness{0.925}
\SetWatermarkScale{2.25}
\SetWatermarkText{\!\!\!\!\!\!\!\!\!\!\!\!\!\!\!\!\!\!\!\!\!\!
{\sc MANUSCRIPT (Submitted)}}


\section{Introduction}\label{intro}
A consequence of G\"odel's Second Incompleteness Theorem is $\Pi_1-$separation of some mathematical theories; for example ${\rm ZFC}$ is not $\Pi_1-$conservative over ${\rm PA}$ since ${\rm ZFC\vdash Con(PA)}$ but (by G\"odel's theorem) ${\rm PA\not\vdash Con(PA)}$, where ${\rm Con}$ is the consistency predicate. Inside ${\rm PA}$, the hierarchy $\{{\rm I\Sigma_n}\}_{n\geqslant 0}$ is not $\Pi_1-$conservative, since ${\rm I\Sigma_{n+1}\vdash Con(I\Sigma_n)}$ (but again ${\rm I\Sigma_{n}\not\vdash Con(I\Sigma_n)}$). As for the bounded arithmetics, we only know that the elementary arithmetic ${\rm I\Delta_0+Exp}$ is not $\Pi_1-$conservative over ${\rm I\Delta_0+\bigwedge_j\Omega_j}$ (see Corollary 5.34 of~\cite{HP98}).  One candidate for $\Pi_1-$separating ${\rm I\Delta_0+Exp}$ from ${\rm I\Delta_0}$ was the Cut-Free Consistency of ${\rm I\Delta_0}$ (see~\cite{PaWi81}): it was already known that ${\rm I\Delta_0+Exp\vdash CFCon(I\Delta_0)}$ and it was presumed that ${\rm I\Delta_0\not\vdash CFCon(I\Delta_0)}$, where ${\rm CFCon}$ stands for Cut-Free Consistency. Though this presumption took rather a long to be established (see~\cite{Wil07}), it opened a new line of research.

The problem of provability (or unprovability) of the cut-free consistency of weak arithmetics is an interesting (double) generalization of G\"odel's Second Incompleteness Theorem:  the theory (being restricted to bounded or weak arithmetics) and also the consistency predicate are both weakened. Here, we do not intend to outline the history of this research line, and refer the reader to~\cite{Sal11a,Sal11b}. Nevertheless, we list some prominent results obtained so far, to put our new result in perspective.

Herbrand Consistency is denoted by ${\rm HCon}$  and (Semantic) Tableau Consistency by ${\rm TabCon}$.
Adamowicz (with Zbierski in 2001~\cite{Ada01} and) in 2002~\cite{Ada02} showed that ${\rm I\Delta_0+\Omega_m\not\vdash HCon(I\Delta_0+\Omega_m)}$ for ${\rm m\geqslant 2}$. She had already shown the unprovability ${\rm I\Delta_0+\Omega_1\not\vdash TabCon(I\Delta_0+\Omega_1)}$ in 1996 (but appeared in 2001 as~\cite{Ada96}). Salehi improved the result of~\cite{Ada02} in~\cite{Sal02} by showing that ${\rm I\Delta_0+\Omega_1\not\vdash HCon(I\Delta_0+\Omega_1)}$ (see also~\cite{Sal11b}) and the result of~\cite{Ada01} in~\cite{Sal01,Sal02} by showing $S\not\vdash{\rm HCon}(S)$ where $S$ is an ${\rm I\Delta_0}-$derivable $\Pi_2-$sentence. This reslt also implied that ${\rm I\Delta_0\not\vdash HCon(\overline{I\Delta_0})}$ holds for a re-axiomatization $\overline{{\rm I\Delta_0}}$ of ${\rm I\Delta_0}$. Willard~\cite{Wil02} showed in 2002 that ${\rm I\Delta_0\not\vdash TabCon(I\Delta_0)}$ and also ${\rm I\Delta_0\not\vdash HCon(I\Delta_0+\Omega_0)}$, where $\Omega_0$ is the axiom of the totality of the squaring function $\Omega_0: \forall x\exists y[y=x\cdot x]$. This was improved in~\cite{Sal11b} by showing ${\rm I\Delta_0\not\vdash HCon(I\Delta_0)}$, without using the $\Omega_0$ axiom. It was also proved in~\cite{Wil02} that $V\not\vdash{\rm HCon}(V)$ for an ${\rm I\Delta_0}-$derivable $\Pi_1-$sentence $V$. Ko{\l}odziejczyk~\cite{Kol06} showed in 2006 that the unprovability ${\rm I\Delta_0+\bigwedge_j\Omega_j\not\vdash HCon(I\Delta_0+\Omega_1)}$ holds; his result was stronger in a sense that it showed ${\rm I\Delta_0+\bigwedge_j\Omega_j\not\vdash HCon(S+\Omega_1)}$ for a finite fragment ${\rm S}\subseteq{\rm I\Delta_0}$.

In this paper we use an idea of an anonymous referee of~\cite{Sal11b} for defining evaluations in a more effective way (Definition~\ref{def-eval}) suitable for bounded arithmetics; this is a great step forward, noting our mentioning  in~\cite{Sal11b} that ``[o]ur definition of Herbrand Consistency is not bet suited for ${\rm I\Delta_0}$". We then partially answer the question proposed by the anonymous referee of~\cite{Sal11a} (see Conjecture~4.1 in~\cite{Sal11a}). The author is grateful to both the referees, for suggestions and inspirations.

We show the existence of a finite fragment $T$ of ${\rm I\Delta_0}$ such that ${\rm I\Delta_0\not\vdash HCon(}T{\rm )}$; this generalizes the result of~\cite{Sal11b}. We also show the existence of an ${\rm I\Delta_0}-$derivable $\Pi_1-$sentence $U$ such that ${\rm I\Delta_0\not\vdash HCon(}U{\rm )}$; this generalizes the main result of~\cite{Sal01,Sal02} and~\cite{Wil02}. For keeping the paper short, and to avoid repeating some technical details, we apologetically invite the reader to consult~\cite{Sal11a,Sal11b}. We also assume familiarity with the Bible of this field~\cite{HP98}.
\section{Herbrand Consistency of Arithmetical Theories}\label{sec:1}
For getting a unique Skolemized formula, it is more convenient to  negation normalize and rectify it.
\begin{definition}[Rectified Negation Normal Form]
A formula is in negation normal form when no implication symbol $\rightarrow$ appears in it, and the negation symbol $\neg$ appears behind the atomic formulas only.
A formula is rectified when different quantifiers refer to different variables and no variable appears both free and bound in the formula.
\hfill$\lozenge\!\!\!\!\!\lozenge$\end{definition}
 Any formula can be uniquely negation normalized by removing the implication connectives (replacing formulas of the form $A\rightarrow B$ with $\neg A\vee B$) and then pushing the negations inside the sub-formulas by de Morgan's Law, until they get to the atomic formulas. Renaming the variables can rectify any formula. Thus one can negation normalize and rectify a formula uniquely, up to a variable renaming.
\begin{definition}[Skolemization]
For any existential formula $\exists x A(x)$ with $m$($\geqslant 0$) free variables, let $\textswab{f}_{\exists x A(x)}$ be a new $m-$ary function symbol (which does not occur in $A$; cf.~\cite{Buss95}). For any rectified negation normal formula $\varphi$ we define $\varphi^S$ inductively:

$\bullet $ \ $\varphi^S=\varphi$ for atomic or negated-atomic formula $\varphi$

$\bullet $ \ $(\varphi\wedge\psi)^S=\varphi^S\wedge\psi^S$

$\bullet $ \ $(\varphi\vee\psi)^S=\varphi^S\vee\psi^S$

$\bullet $ \ $(\forall x\varphi)^S=\forall x\varphi^S$

$\bullet $ \ $(\exists x\varphi)^S=\varphi^S[\textswab{f}_{\exists x\varphi(x)}(\overline{y})/x]$
where $\overline{y}$ are the free variables of $\exists x\varphi(x)$.

\noindent Finally, the Skolemized form $\varphi^{\rm Sk}$ of the formula $\varphi$ is obtained by removing all the (universal) quantifiers of $\varphi^S$. The resulted formula is an open (quantifier-less) formula, with probably some free variables. If those (free) variables are substituted with some ground (variable-free) terms,  we obtain an {Skolem instance} of that formula.
\hfill$\lozenge\!\!\!\!\!\lozenge$\end{definition}
Summing up, to get an Skolem instance of a given formula $\varphi$ we first negation normalize and then rectify it to get a formula $\varphi^{\rm RNNF}$; then we remove the quantifiers of $(\varphi^{\rm RNNF})^S$ to get $(\varphi^{\rm RNNF})^{\rm Sk}$, and substituting its free variables with some ground terms, gives us an Skolem instance of the formula $\varphi$. Let us note that the Skolem instances of a formula are determined uniquely.
\begin{theorem}[Herbrand-Skolem-G\"odel]{
Any theory $T$ is
equi-consistent with its Skolemized theory. In other
words, the theory $T$ is consistent if and only if every finite set of Skolem
instances of $T$ is (propositionally) satisfiable.\hfill\ding{113}
}\end{theorem}
\begin{example}\label{q1example}
In the language of arithmetic  $\mathcal{L}_A=\{\textsf{0},\textsf{S},\textsf{+},\cdot,\leqslant\}$, let ${\rm Ind}_\square$ be the instance of induction principle
$\psi(\textsf{0})\wedge\forall x [\psi(x)\rightarrow\psi(\textsf{S}(x))]\rightarrow\forall x\psi(x)$
for $\psi(x)=\exists y [y\leqslant x\cdot x \wedge y=x\cdot x]$.
This is an axiom of the theory ${\rm I\Delta_0}$.
 Rectified Negation Normal Form $({\rm Ind}_\square)^{\rm RNNF}$ of ${\rm Ind}_\square$ is

 $\forall u [u\not\leqslant\textsf{0}\cdot\textsf{0}\vee u\not=\textsf{0}\cdot\textsf{0}]\ \bigvee\ \exists w\Big[\exists z[z\leqslant w\cdot w\wedge z=w\cdot w]\wedge\forall v[v\not\leqslant \textsf{S}(w)\cdot\textsf{S}(w)\vee v\not=\textsf{S}(w)\cdot\textsf{S}(w)]\Big]\ \bigvee$

 $\forall x\exists y [y\leqslant x\cdot x\wedge y=x\cdot x]$.

\noindent  If  $\textswab{c}$ is the Skolem constant symbol for
$\exists w\big[\exists z[z\leqslant w\cdot w\wedge z=w\cdot w]\wedge\forall v[v\not\leqslant \textsf{S}(w)\cdot\textsf{S}(w)\vee v\not=\textsf{S}(w)\cdot\textsf{S}(w)]\big]$,
 and $\textswab{q}(x)$ is the Skolem function symbol for the formula $\exists z [z\leqslant x\cdot x\wedge z=x\cdot x]$, then $(({\rm Ind}_\square)^{\rm RNNF})^S$  is

 $\forall u [u\not\leqslant\textsf{0}\cdot\textsf{0}\vee u\not=\textsf{0}\cdot\textsf{0}]\ \bigvee\ \big[[\textswab{q}(\textswab{c})\leqslant\textswab{c}\cdot\textswab{c}
  \wedge\textswab{q}(\textswab{c})=\textswab{c}\cdot\textswab{c}]
  \wedge\forall v[v\not\leqslant\textsf{S}(\textswab{c})\cdot\textsf{S}(\textswab{c})\vee v\not=\textsf{S}(\textswab{c})\cdot\textsf{S}(\textswab{c})]\big]\ \bigvee$

 $\forall x [\textswab{q}(x)\leqslant x\cdot x\wedge\textswab{q}(x)=x\cdot x]$.

\noindent Finally, the Skolemized form $({\rm Ind}_\square)^{\rm Sk}$ of $\varphi$ is obtained as:

$[u\not\leqslant\textsf{0}\cdot\textsf{0}\vee u\not=\textsf{0}\cdot\textsf{0}]\ \bigvee\ \big[[\textswab{q}(\textswab{c})\leqslant\textswab{c}\cdot\textswab{c}
  \wedge\textswab{q}(\textswab{c})=\textswab{c}\cdot\textswab{c}]
  \wedge [v\not\leqslant\textsf{S}(\textswab{c})\cdot\textsf{S}(\textswab{c})\vee v\not=\textsf{S}(\textswab{c})\cdot\textsf{S}(\textswab{c})]\big]\ \bigvee$

$[\textswab{q}(x)\leqslant x\cdot x\wedge\textswab{q}(x)=x\cdot x]$.

\noindent Substituting $u/\textsf{0}$, $v/\textsf{S}(\textswab{c})\cdot\textsf{S}(\textswab{c})$, $x/t$ will result in the following Skolem instance of $\varphi$:

$[\textsf{0}\not\leqslant\textsf{0}\cdot\textsf{0}\vee \textsf{0}\not=\textsf{0}\cdot\textsf{0}] \ \bigvee\ \big[[\textswab{q}(\textswab{c})\leqslant\textswab{c}\cdot\textswab{c}
  \wedge\textswab{q}(\textswab{c})=\textswab{c}\cdot\textswab{c}]
  \wedge [\textsf{S}(\textswab{c})\cdot\textsf{S}(\textswab{c})\not\leqslant\textsf{S}(\textswab{c})\cdot\textsf{S}(\textswab{c})
  \vee \textsf{S}(\textswab{c})\cdot\textsf{S}(\textswab{c})\not=\textsf{S}(\textswab{c})\cdot\textsf{S}(\textswab{c})]\big]\ \bigvee$

  $[\textswab{q}(t)\leqslant t\cdot t\wedge\textswab{q}(t)=t\cdot t]$.
\hfill$\lozenge\!\!\!\!\!\lozenge$\end{example}
 Propositional satisfiability is usually arithmetized from the usual provability, only in propositional logic (see e.g. \cite{HP98}); but in a series of more recent papers, this notion have been arithmetized differently, according to ones needs (\cite{Ada96,Ada01,Ada02,Kol06,Sal01,Sal02,Sal11a,Sal11b,Wil02}).
We formalize the notion of propositional satisfiability by means of evaluations (as in the op. cit. papers) on sets of (Skolem) ground terms, but in a more effective way. To get a small evaluation on a given set of terms, we first sort its members, and then require the equality relation to be  a congruence.

We will call the ground terms constructed from Skolem function (and constant) symbols, simply {\em terms}. For a set $A$, its cardinality will be denoted by $|A|$, and for a sequence $s$, its length will be also denoted by $|s|$. The $(i+1)$th member of $s$ is denoted by $(s)_i$ for any $i<|s|$; so $s=\langle(s)_0,(s)_1,\ldots(s)_{|s|-1}\rangle$. Let $\thickapprox$ and $\prec$ be two new symbols, not in the language of arithmetic $\mathcal{L}_A=\langle \textsf{0},\textsf{S},\textsf{+},\cdot,\leqslant\rangle$.

\begin{definition}[Pre-Evaluation]
For a set of terms $\Lambda$ (with $|\Lambda|\geqslant 2$), a pre-evaluation on $\Lambda$ is a sequence $p$ that satisfies the following conditions:

\noindent (1) length of $p$ is  $|p|=2|\Lambda|-1$;

\noindent (2) for any $0\leqslant i\leqslant|\Lambda|-1$ we have $(p)_{2i}\in\Lambda$;

\noindent (3) for any $1\leqslant i\leqslant|\Lambda|-1$ we have $(p)_{2i-1}\in\{\prec,\thickapprox\}$;

\noindent (4) for any term $t\in\Lambda$ there exists a unique $0\leqslant j\leqslant|\Lambda|-1$ such that $(p)_{2j}=t$.
\hfill$\lozenge\!\!\!\!\!\lozenge$
\end{definition}
In other words, a pre-evaluation on $\Lambda$ sorts (organizes) the terms in $\Lambda$, starting from the smallest and ending in the largest.
\begin{example}\label{example1}
A pre-evaluation on $\{\alpha_0,\alpha_1,\alpha_2,\alpha_3,\alpha_4,\alpha_5,\alpha_6\}$ is a sequence like

{$p=\langle \alpha_4,\prec,\alpha_7,\thickapprox,\alpha_1,\thickapprox,
\alpha_5,\prec,\alpha_3,\prec\alpha_6,\thickapprox,\alpha_2\rangle$.}
\hfill$\lozenge\!\!\!\!\!\lozenge$
\end{example}
\begin{definition}[Equality and Order in Pre-Evaluations]
In a pre-evaluation $p$ on $\Lambda$ define the relations $\thickapprox_p$ and $\prec_p$ on $\Lambda^2$ by the following conditions for $s,t\in\Lambda$:

\noindent (1) $s\thickapprox_p t$ if there exists a sub-sequence $q$ of $p$ of length $2l-1$ ($l\geqslant 1$) such that

(a) either ($(q)_0=s \,\&\, (q)_{2l-2}=t$) or ($(q)_0=t \,\&\, (q)_{2l-2}=s$);

(b) for any $1\leqslant i\leqslant l-1$, $(q)_{2i-1}= \ \thickapprox$.

\noindent (2) $s\prec_p t$ if there exists a sub-sequence $q$ of $p$ of length $2l-1$ ($l\geqslant 1$) such that

(a) $(q)_0=s$ and $(q)_{2l-2}=t$;

(b) there exists some $1\leqslant i\leqslant l-1$ for which $(q)_{2i-1}= \ \prec$.
\hfill$\lozenge\!\!\!\!\!\lozenge$
\end{definition}
{\em Example~\ref{example1} (Continued)} We have $\alpha_1\thickapprox_p\alpha_5\thickapprox_p\alpha_7$ and $\alpha_2\thickapprox_p\alpha_6$. Also, $\alpha_4\prec_p\alpha_1$, $\alpha_4\prec_p\alpha_5$, $\alpha_4\prec_p\alpha_7$, $\alpha_1\prec_p\alpha_2$,  $\alpha_1\prec_p\alpha_3$, and $\alpha_1\prec_p\alpha_6$ hold.\hfill$\lozenge\!\!\!\!\!\lozenge$
\begin{lemma}[Equivalence and Order by Pre-Evaluation]\label{equiorder} Let $\Lambda$ be a set of terms, and $p$ be a pre-evaluation on $\Lambda$.

\noindent (1) The relation $\thickapprox_p$ is an equivalence on $\Lambda$.

\noindent (2)  The relation $\prec_p$ is a total order on $\Lambda$.

\noindent (3) The relations $\thickapprox_p$ and $\prec_p$ are compatible with each other: if $t\thickapprox_p s$, and $t\prec_p u$ (respectively, $u\prec_p t$), then $s\prec_p u$ (respectively, $u\prec_p s$).
\end{lemma}
\begin{proof}
 The parts (1) and (2) are immediate. For (3), suppose $t\thickapprox_p s$ and $t\prec_p u$. Then there is a sub-sequence $q$ of $p$ which starts from $t$ and ends with $u$ and contains at least one special symbol $\prec$. There must also be some other sub-sequence $r$ which starts from either $t$ or $s$ and ends with the other one, and all its special symbols are equality $\thickapprox$. If $r$ starts from $s$ (and so ends with $t$), then the concatenation of $r$ and $q$ results in a sub-sequence which starts from $s$ and ends with $u$ and contains some special symbol $\prec$. Whence $s\prec_p u$. And if $r$ starts from $t$, then $q$ cannot be a sub-sequence of $r$ because all the special symbols in $r$ are $\thickapprox$ and $q$ contains at least  one special symbol $\prec$. Thus $r$ has to be a sub-sequence of $q$. Then there must exist a sub-sequence of $p$ which starts from $s$ and ends with $u$ and contains a special symbol $\prec$; whence $s\prec_p u$. The other case ($u\prec_p t$) can be proved very similarly.
 \end{proof}
\begin{definition}[Evaluation]\label{def-eval}
 A pre-evaluation $p$ on a set of terms $\Lambda$ is called an evaluation when, for any term $t,s\in\Lambda$ and any term $u(x)$ with the free variable $x$, if $t\thickapprox_p s$ and $u(t/x),u(s/x)\in\Lambda$ hold, then $u(t/x)\thickapprox_p u(s/x)$ holds too.  \hfill$\lozenge\!\!\!\!\!\lozenge$
\end{definition}
In other words, an evaluation on $\Lambda$ is a pre-evaluation $p$ on $\Lambda$ whose equivalence relation $\thickapprox_p$ is a congruence relation on $\Lambda$.
\begin{definition}[Satisfaction in an Evaluation] Let $\Lambda$ be a set of terms and $p$ an evaluation on it. For terms $t,s\in\Lambda$ we write $p\models t=s$ when $t\thickapprox_p s$ holds. We also write $p\models t\leqslant s$ when either $t\thickapprox_p s$ or $t\prec_p s$ holds.  So, for atomic formulas $\varphi$ in the language of arithmetic $\mathcal{L}_A$ we have defined the notion of satisfaction $p\models\varphi$. The satisfaction relations can be extended to all  open (quantifier-less) formulas as:

$\bullet$\ $p\models \varphi\wedge\psi \iff p\models\varphi \textrm{ and } p\models\psi$

$\bullet$\ $p\models \varphi\vee\psi \iff p\models\varphi \textrm{ or } p\models\psi$

$\bullet$\ $p\models \varphi\rightarrow\psi \iff \textrm{ if } p\models\varphi \textrm{ then } p\models\psi$

$\bullet$\ $p\models \neg\varphi \iff  p\not\models\varphi$\hfill$\lozenge\!\!\!\!\!\lozenge$
\end{definition}
\begin{lemma}[Leibniz's Law]
Any evaluation $p$ on any set of terms $\Lambda$ satisfies all the available Skolem instances of the axioms of equational logic, in particular Leibniz's Law: for any $t,s\in\Lambda$ and any open formula $\varphi(x)$, we have $p\models t=s\wedge\varphi(t)\rightarrow\varphi(s)$.
\end{lemma}
\begin{proof}
Suppose $p\models t=s$. By induction on (the complexity) of (the open formula) $\varphi$ one can show that $p\models\varphi(t)$ if and only if $p\models\varphi(s)$. For atomic $\varphi$ it follows from Lemma~\ref{equiorder}, and for the more complex formulas it follows from the inductive definition of satisfaction in evaluations.
\end{proof}
\begin{definition}[$T-$evaluation on $\Lambda$]
For a set of terms $\Lambda$, an Skolem instance of a formula is called to be available in $\Lambda$ if all the terms appearing in it belong to $\Lambda$.
For a theory $T$ and a set of terms $\Lambda$ and an evaluation $p$ on $\Lambda$, we say that $p$ is an $T-$evaluation on $\Lambda$ if $p$ satisfies every Skolem instance of every sentence in $T$ which is available in $\Lambda$.
\hfill$\lozenge\!\!\!\!\!\lozenge$\end{definition}
So, $T-$evaluations, for a theory $T$, are kind of partial models of $T$.   Indeed,  if $\Lambda$ is the set of all (ground) terms (constructed from the language of $T$ and the Skolem function symbols of the axioms of $T$), then any $T-$evaluaton on $\Gamma$ (if exists) is a {\em Herbrand Model} of $T$. Herbrand's Theorem can be read as
``{\em A theory $T$ is consistent if and only if  for every finite set of (Skolem) terms, there exists an $T-$evaluation on it.}"
 Thus, the notion of {\em Herbrand Consistency} of a theory $T$ is (equivalent to) the existence of an $T-$evaluation on any (finite) set of terms.
\begin{example}\label{q2example}
Let $T$ be axiomatized by the following sentences in  $\mathcal{L}_A$:

$\bullet$\ $\forall x [x\cdot\textsf{0}=\textsf{0}]$;

$\bullet$\ $\exists y\leqslant\textsf{0}\cdot\textsf{0}[y=\textsf{0}\cdot\textsf{0}]\wedge
\forall x\big[\exists y\leqslant x\cdot x[y=x\cdot x]\rightarrow\exists y\leqslant\textsf{S}(x)\cdot\textsf{S}(x)[y=\textsf{S}(x)\cdot\textsf{S}(x)]\big]\rightarrow\forall x\exists y\leqslant x\cdot x[y=x\cdot x]$.

\noindent Let $\Lambda=\{\textsf{0},\textsf{0}\cdot\textsf{0},\textswab{c},
\textswab{c}\cdot\textswab{c},\textswab{q}(\textswab{c}),
\textsf{S}(\textswab{c})\cdot\textsf{S}(\textswab{c}),t,t\cdot t,\textswab{q}(t)\}$ where $\textswab{c}$ and $\textswab{q}$ are as in Example~\ref{q1example}. As we saw in that example, the following is an instance of the the second axiom (${\rm Ind}_\square$), which is also available in $\Lambda$:

$[\textsf{0}\not\leqslant\textsf{0}\cdot\textsf{0}\vee \textsf{0}\not=\textsf{0}\cdot\textsf{0}]\ \bigvee$

$\big[[\textswab{q}(\textswab{c})\leqslant\textswab{c}\cdot\textswab{c}
  \wedge\textswab{q}(\textswab{c})=\textswab{c}\cdot\textswab{c}]
  \wedge [\textsf{S}(\textswab{c})\cdot\textsf{S}(\textswab{c})\not\leqslant\textsf{S}(\textswab{c})\cdot\textsf{S}(\textswab{c})
  \vee \textsf{S}(\textswab{c})\cdot\textsf{S}(\textswab{c})\not=\textsf{S}(\textswab{c})\cdot\textsf{S}(\textswab{c})]\big]\ \bigvee$

$[\textswab{q}(t)\leqslant t\cdot t\wedge\textswab{q}(t)=t\cdot t]$.

\noindent Suppose $p$ is an $T-$evaluation on $\Lambda$. By the first axiom $p$ must satisfy the instance $\textsf{0}\cdot\textsf{0}=\textsf{0}$, so we should have $p\models \textsf{0}\cdot\textsf{0}=\textsf{0}$. Thus, $p$ cannot satisfy the first disjunct of the above instance. Indeed, $p$ cannot satisfy the second disjunct either, because for any term $u$ we have $p\models u\leqslant u\wedge u=u$.  Thus, $p$ cannot satisfy the second conjunct of the second disjunct. Whence, $p$ must satisfy the third disjunct of the above instance, and in particular we should have $p\models\textswab{q}(t)=t\cdot t$.
\hfill$\lozenge\!\!\!\!\!\lozenge$\end{example}
\begin{definition}[Skolem Hull]\label{defshull}
Let $\mathcal{L}_A^{\rm Sk}$ be the language expanding $\mathcal{L}_A$ by the Skolem function (and constant) symbols of all the existential formulas in the language $\mathcal{L}_A$. Or in other words,  $\mathcal{L}_A^{\rm Sk}$ is the set $\mathcal{L}_A^{\rm Sk}=\{\textswab{f}_{\exists x\varphi(x)}\mid\varphi\textrm{ is an }\mathcal{L}_A-\textrm{formula}\}.$
\noindent For a given set of terms $\Lambda$,
let $\Lambda^{\langle j\rangle}$ be defined by induction on $j$:

\noindent $\Lambda^{\langle 0\rangle}=\Lambda$, and

\noindent $\Lambda^{\langle j+1\rangle}=\Lambda^{\langle j\rangle}\cup\{f(t_1,\ldots,t_m)\mid f\!\in\!\mathcal{L}\wedge t_1,\ldots,t_m\!\in\!\Lambda^{\langle j\rangle}\}\cup\{\textswab{f}_{\exists x\varphi(x)}(t_1,\ldots,t_m)\mid \ulcorner\varphi\urcorner\leqslant j\wedge t_1,\ldots,t_m\!\in\!\Lambda^{\langle j\rangle}\}$,

\noindent where $\ulcorner\varphi\urcorner$ is the G\"odel code of $\varphi$.
\hfill$\lozenge\!\!\!\!\!\lozenge$\end{definition}
Bounding the G\"odel code of $\varphi$ in the above definition will enable us to have some efficient (upper bound) for the G\"odel code of $\Lambda^{\langle j\rangle}$  (see~\cite{Sal11a,Sal11b}).

Herbrand's theorem implies that for any $\exists_1-$formula $\exists x\psi(x)$ (where $\psi$ is an open formula) and any theory $T$, if  $T\vdash\exists x\psi(x)$ then there are some (Skolem) terms $t_1,\ldots,t_n$ such that $T^{\rm Sk}\vdash\psi(t_1)\vee\ldots\vee\psi(t_n)$. Usually this observation is called Herbrand's Theorem. We will need a somehow dual of this fact.
\begin{lemma}[Herbrand Proof of Universal Formulas]\label{lem-univ} For a $\forall_1-$formula $\forall x\psi(x)$ (where $\psi$ is open) and a theory $T$, suppose $T\vdash\forall x\psi(x)$. Let $\Lambda$ be a set of terms and $t\in\Lambda$. There exists a finite (standard) $k\geqslant 0$ such that for any $T-$evaluation $p$ on $\Lambda^{\langle k\rangle}$ we have $p\models\psi(t)$.
\end{lemma}
\begin{proof} By $T\vdash\forall x\psi(x)$ the theory $T^{\rm Sk}\cup\{\neg\psi(\textswab{c})\}$, where $\textswab{c}$ is the Skolem constant symbol for $\exists x\neg\psi(x)$, is inconsistent. Suppose $\varphi$ is the rectified negation normal form of $\neg\psi$. Then, by Herbrand's theorem, there exists some finite set of terms $\Gamma$ such that there can be no $(T^{\rm Sk}\cup\{\varphi(\textswab{c})\})-$evaluation on it. Since $\textswab{c}$ appears in $\Gamma$ we write it as $\Gamma(\textswab{c})$, and by $\Gamma(t)$ we denote the set of terms which result from the terms of $\Gamma(\textswab{c})$ by replacing $\textswab{c}$ with $t$ everywhere. It can be clearly seen that there exists some $k\in\mathbb{N}$ such that $\Gamma(t)\subseteq\Lambda^{\langle k\rangle}$. Whence, there cannot be any $(T^{\rm Sk}\cup\{\varphi(t)\})-$evaluation on $\Lambda^{\langle k\rangle}$. Thus, any $T-$evaluation $p$ on $\Lambda^{\langle k\rangle}$ must satisfy $p\not\models\varphi(t)$ or $p\models\psi(t)$.
\end{proof}
\begin{example}\label{q3example}
Let the theory $T$, in the language of arithmetic $\mathcal{L}_A$, be axiomatized by

\begin{tabular}{ll}
$(1) \,\forall x [\textsf{S}(x)\not=\textsf{0}]$ & $(2)\, \forall x,y [x+\textsf{S}(y)=\textsf{S}(x+y)]$ \\
$(3)\,\forall x \exists z [x\not=\textsf{0}\rightarrow x=\textsf{S}(z)]$ & $(4)\,\forall x,y\exists z [x\leqslant y\rightarrow z+x=y]$
\end{tabular}

\noindent
For the open formula $\psi(x)= (x\leqslant\textsf{0}\rightarrow x=\textsf{0})$ we have $T\vdash\forall x\psi(x)$.

Let $\textswab{p}(x)$ be the Skolem function for the formula $\exists z [x=\textsf{0}\vee x=\textsf{S}(z)]$, and $\textswab{h}(x,y)$ be the Skolem function for the formula $\exists z [x\not\leqslant y\vee z+x=y]$. Then the Skolemized form $T^{\rm Sk}$ of the theory $T$ will be as:

\begin{tabular}{ll}
 $(1')\,\textsf{S}(x)\not=\textsf{0}$ & $(2')\,x+\textsf{S}(y)=\textsf{S}(x+y)$ \\
$(3')\,x=\textsf{0}\vee x=\textsf{S}(\textswab{p}(x))$ & $(4')\,x\not\leqslant y\vee \textswab{h}(x,y)+x=y$
\end{tabular}

\noindent For a fixed term $t$ let $\Gamma_t$ be the following set of terms:
\newline\centerline{$\Gamma_t=\{\textsf{0},t,\textswab{h}(t,\textsf{0}),
\textswab{h}(t,\textsf{0})+t,\textswab{p}(t),\textsf{S}(\textswab{p}(t)),
\textswab{h}(t,\textsf{0})+\textswab{p}(t),\textswab{h}(t,\textsf{0})+\textsf{S}(\textswab{p}(t)),
\textsf{S}\big(\textswab{h}(t,\textsf{0})+\textswab{p}(t)\big)\}$.}
\noindent Now we show that any $T-$evaluation $p$ on $\Gamma_t$ must satisfy $p\models\psi(t)$ or, equivalently, if $p\models t\leqslant\textsf{0}$ then $p\models t=\textsf{0}$. Assume $p\models t\leqslant\textsf{0}$. Then by the fourth axiom we have $p\models\textswab{h}(t,\textsf{0})+t=\textsf{0}$. If $p\models t=\textsf{0}$ does not hold, then $p\models t\not=\textsf{0}$, so by the third axiom we have $p\models t=\textsf{S}(\textswab{p}(t))$. Whence, $p\models\textswab{h}(t,\textsf{0})+\textsf{S}(\textswab{p}(t))=\textsf{0}$. On the other hand, by the second axiom, $p\models\textswab{h}(t,\textsf{0})+\textsf{S}(\textswab{p}(t))=
\textsf{S}\big(\textswab{h}(t,\textsf{0})+\textswab{p}(t)\big)$. So, we infer that $p\models\textsf{S}\big(\textswab{h}(t,\textsf{0})+\textswab{p}(t)\big)=\textsf{0}$,
which is in contradiction with the first axiom. Thus, $p\models t=\textsf{0}$ must hold, which shows that $p\models\psi(t)$.
\hfill$\lozenge\!\!\!\!\!\lozenge$\end{example}
As was mentioned before, for a consistent theory $T$ there must exist some Herbrand Model of $T$.
\begin{definition}[Definable Herbrand Models]\label{defhmodel}
Let $\Lambda$ be a set of terms, and define its Skolem Hull to be  $\Lambda^{\langle\infty\rangle}=\bigcup_{n\in\mathbb{N}}\Lambda^{\langle n\rangle}$ (see Definition~\ref{defshull}). For  an evaluation $p$ on $\Lambda^{\langle\infty\rangle}$, let $\mathfrak{M}(\Lambda,p)=\{t/p\mid t\in\Lambda^{\langle\infty\rangle}\}$, where $t/p$ is the equivalence class of the relation $\thickapprox_p$ containing $t$ (cf. Lemma~\ref{equiorder}).  Put the structure

 (1) $f^{\mathfrak{M}(\Lambda,p)}(t_1/p,\ldots,t_m/p)=f(t_1,\ldots,t_m)/p$,

 (2) $R^{\mathfrak{M}(\Lambda,p)}=\{(t_1/p,\ldots,t_m/p)\mid p\models R(t_1,\ldots,t_m)\}$,

\noindent on $\mathfrak{M}(\Lambda,p)$, for any $m-$ary function symbol $f$ and any $m-$ary relation symbol $R$.
\hfill$\lozenge\!\!\!\!\!\lozenge$\end{definition}
\begin{lemma}[Herbrand Models by Evaluations]\label{lem-hmodel} The  structure on $\mathfrak{M}(\Lambda,p)$ is well-defined, and for a theory $T$, if $p$ is an $T-$evaluation on $\Lambda$ then $\mathfrak{M}(\Lambda,p)\models T$.
\hfill\ding{113}\end{lemma}
\section{Bounded Arithmetic and Herbrand Consistency}
 By an efficient G\"odel coding (see e.g. Chapter V of \cite{HP98}) we can code sets, sequences (and so the syntactic concepts like Skolem function symbols, Skolem instances, evaluations, etc.) such that the following (\cite{HP98}) hold for any sequences $\alpha,\beta$:
\begin{itemize}
\item $\ulcorner\alpha\ast\beta\urcorner\leqslant 64\cdot(\ulcorner\alpha\urcorner\cdot\ulcorner\beta\urcorner)$, where $\ast$ denotes concatenation;
\item $|\alpha|\leqslant\log(\ulcorner\alpha\urcorner)$.
\end{itemize}
It follows that for any sets $A,B$ we  have $\ulcorner A\cup B\urcorner\leqslant 64\cdot(\ulcorner A\urcorner\cdot\ulcorner B\urcorner)$ and $|A|\leqslant\log(\ulcorner A\urcorner)$.
We write $X\in\mathcal{O}(Y)$ to indicate that $X\leqslant Y\cdot n+n$ for some $n\in\mathbb{N}$; that is $X$ is linearly bounded by $Y$. The above (efficient) coding has the property that for any sequence $U=\langle u_1,\ldots,u_l\rangle$ we have $\log(\ulcorner U\urcorner)\in\mathcal{O}(\sum_{i}\log(\ulcorner u_i\urcorner))$. For any evaluation $p$ on a set of terms $\Lambda$ it can be seen that $\log (\ulcorner p\urcorner)\in\mathcal{O}(\log(\ulcorner\Lambda\urcorner))$.

Let us note that all of the concepts introduced so far can be formalized in the language of arithmetic $\mathcal{L}_A$.  Here we make the observation that,  having an arithmetically definable set of terms $\Lambda$, the sets $\Lambda^{\langle j\rangle}$ are all definable in arithmetic (in terms of $\Lambda$ and $j$), but the set $\Lambda^{\langle\infty\rangle}$ is not definable by an arithmetical formula. We will come to this point later.
The arithmetical theory we are interested here is denoted by ${\rm I\Delta_0}$ which is usually axiomatized by Robinson's arithmetic,
 in the language  $\mathcal{L}_A$, plus the induction axiom for
 bounded formulas (see e.g. \cite{HP98}).

 In this section we prove our main result: the existence of a finite fragment $T~\subseteq~{\rm I\Delta_0}$ whose Herbrand Consistency is not provable in ${\rm I\Delta_0}$. As the exponential function $x~\mapsto~2^x$ is not available (provably total) in ${\rm I\Delta_0}$, then we denote by $\textfrak{log}$ the set of  elements $x$ for which $\exp(x)=2^x$ exists. Let us note that for a model $\mathcal{M}$, the set $\textfrak{log}(\mathcal{M})$ is the logarithm of the elements of $\mathcal{M}$. The set $\textfrak{log}$ is closed under $\textsf{S}$ and $\textsf{+}$, but not under ${\times}$, in ${\rm I\Delta_0}$. We will use the term {\em cut} for any definable and downward closed set (not necessarily closed under $\textsf{S}$) in the arithmetical models. The formula $``y=\exp(x)"$ is expressible in $\mathcal{L}_A$, and ${\rm I\Delta_0}$ can prove some of the basic properties of $\exp$ (cf. \cite{HP98}), though cannot prove its totality: ${\rm I\Delta_0}\not\vdash\forall x\exists y [y=\exp(x)]$. By $\textfrak{log}^2$ we denote the set of elements $x$ for which $\exp^2(x)=2^{2^x}$ exists; the superscripts on top of the functions denote the iteration. Similarly, $\textfrak{log}^n=\{x\mid\exists y[y=\exp^n(x)]\}$, where $\exp^n$ denotes the $n$ time iteration of the exponential function $\exp$.

We use a deep theorem in bounded arithmetic, which happens to be the very last theorem of \cite{HP98}. It reads, in our terminology, as:
\begin{center}
{\em For any $k\geqslant 0$ there exists a bounded formula $\varphi(x)$ such that

${\rm I\Delta_0+\Omega_1}\vdash\forall x\in\textfrak{log}^{k+1}\varphi(x)$, \,but\, ${\rm I\Delta_0+\Omega_1}\not\vdash\forall x\in\textfrak{log}^{k}\varphi(x)$.}
\end{center}

It can be clearly seen that the theorem also holds for ${\rm I\Delta_0}$  instead of ${\rm I\Delta_0+\Omega_1}$, and for any cut $I$ (and its logarithm $\log I=\{x\mid\exists y\in I[y=\exp(x)]\}$) instead of $\textfrak{log}^k$ (and its logarithm $\textfrak{log}^{k+1}$); see also~\cite{Ada02}~and (Theorem~3.6 of)~\cite{Sal11a}.
\begin{theorem}[${\rm \Pi_1}-$Separation of Logarithmic Cuts]\label{th-ilogi}
For any cut $I$ there exists a bounded formula $\varphi(x)$ such that  ${\rm I\Delta_0}\cup\{\exists x\!\in\! I\ \varphi(x)\}$ is consistent, but  ${\rm I\Delta_0}\cup\{\exists x\!\in\!\log I\ \varphi(x)\}$ is not consistent.
\hfill\ding{113}\end{theorem}
We will find the desired finite fragment of ${\rm I\Delta_0}$ (whose Herbrand Consistency is not provable in ${\rm I\Delta_0}$) in three steps (the following subsections) before proving the main result (in the last subsection). For doing so, we will show that for sufficiently strong finite fragments of ${\rm I\Delta_0}$, like $T$, if ${\rm I\Delta_0}\vdash{\rm HCon}(T)$ then the consistency of the theory ${\rm I\Delta_0}\cup\{\exists x\!\in\! I\ \theta(x)\}$, for some suitable cut $I$ and a suitable bounded formula $\theta$, implies the consistency of the theory $T\cup\{\exists x\!\in\!\log I\ \theta(x)\}$. As we will see, this contradicts Theorem~\ref{th-ilogi}.
\subsection{The First Finite Fragment}
Assuming the consistency of the theory ${\rm I\Delta_0}\cup\{\exists x\!\in\! I\ \varphi(x), {\rm HCon}(T)\}$, and inconsistency of the theory $T\cup\{\exists x\!\in\!\log I\ \varphi(x)\}$, we can construct a model $\mathfrak{M}$, from a given model $\mathcal{M}\models{\rm I\Delta_0}\cup\{\exists x\!\in\! I\ \varphi(x), {\rm HCon}(T)\}$, such that
$\mathfrak{M}\models T\cup\{\exists x\!\in\!\log I\ \varphi(x)\}$; which is in contradiction with the assumptions. For that, let us take a (hypothetical) model $\mathcal{M}\models{\rm I\Delta_0}\cup\{a\!\in\! I\wedge\varphi(a)\}\cup \{{\rm HCon}(T)\}$ for some $a\in\mathcal{M}$. Then we form the set
$\Gamma=\{\underline{0},\underline{1},\underline{2},\ldots\underline{\omega_1(a)}\}$ where $\underline{i}$ is a term in $\mathcal{L}_A$ representing the number $i$, defined inductively as $\underline{0}=\textsf{0}$ and $\underline{i+1}=\textsf{S}(\underline{i})$. From the assumption $\mathcal{M}\models{\rm HCon}(T)$ we find an $T-$evaluation $p$ on $\Lambda^{\langle j\rangle}$, for a suitable $j$ and a suitable $\Lambda$ which contains the above set $\Gamma$. Then we can form the model $\mathfrak{M}(\Lambda,p)$ and, by some technical details, show that $\mathfrak{M}(\Lambda,p)\models T+\exists x\in\log I\varphi(x)$. The bound $\omega_1(a)$ assures us that the set $\Gamma$ contains the range of (the bounded) quantifiers in the (bounded) formula $\varphi(a)$. For the G\"odel code of $\underline{i}$ we have $\log(\ulcorner\underline{i}\urcorner)\in\mathcal{O}(\log(2^i))$ and so   $\log(\ulcorner\Gamma\urcorner)\in
\mathcal{O}(\log(2^{(\omega_1(a))^2}))$ whence
$\log(\ulcorner\Gamma\urcorner)\in\mathcal{O}\big(\log\big(\exp^2(2(\log a)^2)\big)\big)$. We need the closure of $\Gamma$ under the Skolem function symbols of (a finite fragment of) $\rm{I\Delta_0}$, that is $\Gamma^{\langle\infty\rangle}$ (see Definitions~\ref{defhmodel}~and~\ref{defshull}). Since, unfortunately, that set is not definable, we consider the set $\Gamma^{\langle j\rangle}$ for a non-standard $j$, which makes sense if $\ulcorner\Gamma\urcorner$  (and so $a$) is non-standard. In case $a$ is standard, then the proof becomes trivial (see below). For some non-standard $j$ with $j\leqslant\log^4(\ulcorner\Gamma\urcorner)$ we can form the set $\Gamma^{\langle j\rangle}$,  in case $\omega_2(\ulcorner\Gamma\urcorner)$ exists (see~\cite{Sal11a,Sal11b}). And finally we have $\log\big (\omega_2(\ulcorner\Gamma\urcorner)\big)\in
\mathcal{O}\big(\log\big (\exp^2(4(\log a)^4)\big)\big)$.
\begin{definition}[The Cut ${\mathcal I}$]\label{def-cuti}
The cut $\mathcal{I}$ is defined to be $\{x\mid\exists  y[y=\exp^2(4(\log a)^4)]\}$, and its logarithm is $\log\mathcal{I}=\{x\mid\exists y[y=\exp^2(4a^4)]\}$.
\hfill$\lozenge\!\!\!\!\!\lozenge$\end{definition}
 Applying theorem~\ref{th-ilogi} to the cut $\mathcal{I}$ defined above, we find a (fixed) bounded formula $\theta$ and a finite fragment $T_0\subseteq{\rm I\Delta_0}$ such that the theory the theory ${\rm I\Delta_0}\cup\{\exists x\!\in\!\mathcal{I}\theta(x)\}$ is consistent, but  $T_0\cup\{\exists x\!\in\!\log\mathcal{I}\theta(x)\}$ is not consistent.
\begin{definition}[The First Fragment $T_0$]\label{def-t0}
Let $T_0$ be a finite fragment of ${\rm I\Delta_0}$ for which there exists a (fixed) bounded formula $\theta$ such that the theory ${\rm I\Delta_0}\cup\{\exists x\!\in\!\mathcal{I}\theta(x)\}$ is consistent, but  $T_0\cup\{\exists x\!\in\!\log\mathcal{I}\theta(x)\}$ is not consistent. Let $\mathcal{M}$ be a (fixed) model such that $\mathcal{M}\models{\rm I\Delta_0}\cup\{\exists x\!\in\!\mathcal{I}\theta(x)\}$.
\hfill$\lozenge\!\!\!\!\!\lozenge$\end{definition}
In the rest of the paper we will show that for a finite fragment $T$ of ${\rm I\Delta_0}$ extending $T_0$ we have that $\mathcal{M}\not\models{\rm HCon}(T)$, where ${\rm HCon}$ is the predicate of Herbrand Consistency.
\subsection{The Second Finite Fragment}
The proof of the main result goes roughly as follows: if $\mathcal{M}\models{\rm HCon}(T)$, for a finite fragment $T\subseteq{\rm I\Delta_0}$ to be specified later, then there exists (in $\mathcal{M}$) some $T-$evaluation $p$ on some $\Lambda^{\langle j\rangle}$, where $\Lambda\supseteq\Gamma$ is to be specified later and $\Gamma$ and $j$ are as in the previous subsection. Whence we can form the model $\mathfrak{M}(\Lambda,p)$, for which we already have $\mathfrak{M}(\Lambda,p)\models T$. Our second finite fragment $T_1$ will have the property that if $T\supseteq T_1$ then $\mathfrak{M}(\Lambda,p)\models\theta_0(\underline{a}/p)$.  The third finite fragment $T_2$ will have the property that if $T\supseteq T_2$ then we have $\mathfrak{M}(\Lambda,p)\models\underline{a}/p\!\in\!\log\mathcal{I}$. So, finally we will get the model $\mathfrak{M}(\Lambda,p)$ which satisfies $\mathfrak{M}(\Lambda,p)\models T+[\underline{a}/p\!\in\!\log\mathcal{I}\wedge\theta_0(\underline{a}/p)]$, or, in the other words, $\mathfrak{M}(\Lambda,p)\models T\cup\{\exists x\!\in\!\log\mathcal{I}\theta_0(x)\}$ which is in contradiction with (the choice of the first finite fragment) $T_0\subseteq T$.
\begin{definition}[The Second Fragment $T_1$]\label{def-t1}
Let $T_1$ be a finite fragment of ${\rm I\Delta_0}$ which can prove the following (${\rm I\Delta_0}-$provable $\forall^\ast-$)sentences:

\begin{tabular}{ll}
$\bullet\ \ x+\textsf{0}=x$ &  $\bullet \ \  x+\textsf{S}(y)=\textsf{S}(x+y)$  \\
$\bullet \ \ x\cdot\textsf{0}=\textsf{0}$ &   $\bullet \ \  x\cdot\textsf{S}(y)=x\cdot y+ x$  \\
$\bullet \ \ x\leqslant\textsf{0}\leftrightarrow x=\textsf{0}$
&  $\bullet \ \  x\leqslant\textsf{S}(y)\leftrightarrow x=\textsf{S}(y)\vee
x\leqslant y$  \\
$\bullet \ \ x\leqslant y\vee y\leqslant x$ &  $\bullet \ \  x\leqslant y\leqslant z\rightarrow x\leqslant z$  \\
$\bullet \ \ x\leqslant z+x$ &  $\bullet \ \  x\leqslant x+z$  \\
$\bullet \ \ x+z\leqslant y+z\rightarrow x\leqslant y$ &  $\bullet \ \  z\not=\textsf{0}\wedge x\cdot z\leqslant y\cdot z\rightarrow x\leqslant y$  \\
$\bullet \ \ x\not=y\leftrightarrow\textsf{S}(x)\leqslant y\vee\textsf{S}(y)\leqslant x$ &  $\bullet \ \  x\not\leqslant y\leftrightarrow\textsf{S}(y)\leqslant x$  \\
\end{tabular}

\noindent and also can prove the following (${\rm I\Delta_0}-$provable $\forall^\ast\exists^\ast-$)sentences:

$\bullet $\ $x\leqslant y\rightarrow\exists z[z+x=y]$

$\bullet$\ $y\not=\textsf{0}\rightarrow\exists q,r [x=r+q\cdot y\wedge r\leqslant y]$\hfill $\lozenge\!\!\!\!\!\lozenge$
\end{definition}
\begin{remark}
It can be seen that $T_1$ can prove the following arithmetical sentences:

\begin{tabular}{ll}
$\bullet \ \ \textsf{S}(x)\not=\textsf{0}$ &  $\bullet \ \   \textsf{S}(x)=\textsf{S}(y)\rightarrow x=y$ \\
$\bullet \ \ \textsf{S}(x)\not\leqslant x$ &   $\bullet \ \ x\not=\textsf{0}\rightarrow\exists y[x=\textsf{S}(y)]$  \\
\end{tabular}

\noindent For a proof, first note that by $x\leqslant y\vee y\leqslant x$ we have $\forall u[u\leqslant u]$, and also from $x\leqslant z+x$ and $x+\textsf{0}=x$ we get $\forall u[\textsf{0}\leqslant u]$. Now, if $\textsf{S}(u)=\textsf{0}$, then $\textsf{S}(u)\leqslant\textsf{0}$, and so by the axiom $x\not\leqslant y\leftrightarrow\textsf{S}(y)\leqslant x$ we get $\textsf{0}\not\leqslant u$, contradiction! Also from the same axiom it follows that $u\not\leqslant u\leftrightarrow\textsf{S}(u)\leqslant u$, and thus $\textsf{S}(u)\not\leqslant u$. If $\textsf{S}(u)=\textsf{S}(v)$ and $u\not=v$ then by $x\not=y\leftrightarrow\textsf{S}(x)\leqslant y\vee\textsf{S}(y)\leqslant x$ we have either $\textsf{S}(u)\leqslant v$ or $\textsf{S}(v)\leqslant u$. If $\textsf{S}(u)\leqslant v$ then $\textsf{S}(v)\leqslant v$, contradiction! The other case is similar. Finally, assume $u\not=\textsf{0}$. Then by $x\leqslant\textsf{0}\leftrightarrow x=\textsf{0}$ we have $u\not\leqslant\textsf{0}$ and so the axiom $x\not\leqslant y\leftrightarrow\textsf{S}(y)\leqslant x$ implies that $\textsf{S}(\textsf{0})\leqslant u$. Thus, by $x\leqslant y\rightarrow\exists z[z+x=y]$ we have $v+\textsf{S}(\textsf{0})=u$ for some $v$. Then from $x+\textsf{S}(y)=\textsf{S}(x+y)$ and $x+\textsf{0}=x$ we conclude that $\textsf{S}(v)=u$. \qquad Q.E.D
\hfill $\lozenge\!\!\!\!\!\lozenge$\end{remark}
The main property of $T_1$ is the following:
\begin{theorem}[The Main Property of $T_1$]\label{th-t1}
Suppose $\mathcal{M}\models{\rm I\Delta_0}+[a\!\in\!\mathcal{I}\wedge\theta(a)]+ {\rm HCon}(T)$ is a non-standard model where $\theta$ is a bounded formula and $a\in\mathcal{M}$ is non-standard and $T\vdash T_1$. If $p\!\in\!\mathcal{M}$ is an $T-$evaluation on $\Lambda^{\langle j\rangle}$ where $\Lambda$ is a set of terms such that $\Lambda\supseteq\Gamma=\{\underline{i}\mid i\leqslant\omega_1(a)\}$ and $j$ is a non-standard element of $\mathcal{M}$, then for any bounded formula $\varphi(x_1,\ldots,x_n)$ and any elements $i_1,\ldots,i_n\leqslant a$, $\mathcal{M}\models\varphi(i_1,\ldots,i_n) \iff
\mathfrak{M}(\Lambda,p)\models\varphi(\underline{i_1}/p,\ldots,\underline{i_n}/p)$.
\end{theorem}
We prove the theorem by induction on (the complexity) of $\varphi$ (see also \cite{Sal11a,Sal11b}).
\begin{lemma}[Another Property of $T_1$]\label{termt1}
Suppose $\mathcal{K}\models T_1$ and $a\!\in\!\mathcal{K}$, and let $t$ be a term in $\mathcal{L}_A$. For any $i_1,\ldots,i_n\leqslant a$ in $\mathcal{K}$ and any $b\!\in\!\mathcal{K}$, if $\mathcal{K}\models b\leqslant t(i_1,\ldots,i_n)$ then there exists a term $s$ and there are some $j_1,\ldots,j_m\leqslant a$ such that $\mathcal{K}\models b=s(j_1,\ldots,j_m)$.
\end{lemma}
\begin{proof}
By induction on $t$:

$\bullet $\ $t=\textsf{0}$: if $\mathcal{K}\models b\leqslant\textsf{0}$ then by the $T_1-$axiom $x\leqslant\textsf{0}\leftrightarrow x=\textsf{0}$ we have $\mathcal{K}\models b=\textsf{0}$.

$\bullet $\ $t=\textsf{S}(t_1)$: if $\mathcal{K}\models b\leqslant\textsf{S}(t_1)$ then by  $x\!\leqslant\!\textsf{S}(y)\leftrightarrow x\!=\!\textsf{S}(y)\vee
x\!\leqslant\!y$ which is a $T_1-$axiom, we have $\mathcal{K}\models b=\textsf{S}(t_1)\vee b\leqslant t_1$, and the result follows from the induction hypothesis.

$\bullet $\ $t=t_1+t_2$: if $\mathcal{K}\models b\leqslant t_1+t_2$ then by the $T_1-$axiom $x\leqslant y\vee y\leqslant x$ we have that $\mathcal{K}\models b\leqslant t_2\vee t_2\leqslant b$. If $\mathcal{K}\models b\leqslant t_2$ then the conclusion follows from the induction hypothesis. Otherwise if $\mathcal{K}\models t_2\leqslant b$ then by $x\leqslant y\rightarrow\exists z[z+x=y]$ (another $T_1-$axiom) there exists some $d\in\mathcal{K}$ such that $\mathcal{K}\models d+t_2=b$. Thus $\mathcal{K}\models d+t_2\leqslant t_1+t_2$, whence by the $T_1-$axiom $x+z\leqslant y+z\rightarrow x\leqslant y$ we have $\mathcal{K}\models d\leqslant t_1$, and the desired result follows from the induction hypothesis and the fact that $\mathcal{K}\models b=d+t_2$.

$\bullet $\ $t=t_1\cdot t_2$: assume $\mathcal{K}\models b\leqslant t_1\cdot t_2$. If $\mathcal{K}\models t_2=\textsf{0}$ then $\mathcal{K}\models t_1\cdot t_2=\textsf{0}$ by the $T_1-$axiom $x\cdot\textsf{0}=\textsf{0}$. And so $\mathcal{K}\models b\leqslant \textsf{0}$ is reduced to the first case above. Now suppose $\mathcal{K}\models t_2\not=\textsf{0}$. Then by the $T_1-$axiom $y\not=\textsf{0}\rightarrow\exists q,r [x=r+q\cdot y\wedge r\leqslant y]$ we have $\mathcal{K}\models b=r+q\cdot t_2 \wedge r\leqslant t_2$ for some $q,r\!\in\!\mathcal{K}$. By the $T_1-$axiom $x\leqslant z+x$ we have $\mathcal{K}\models q\cdot t_2\leqslant r+q\cdot t_2=b\leqslant t_1\cdot t_2$ and so from the $T_1-$axiom $x\leqslant y\leqslant z\rightarrow x\leqslant z$ it follows that $\mathcal{K}\models q\cdot t_2\leqslant t_1\cdot t_2$, and the $T_1-$axiom  $z\not=\textsf{0}\wedge x\cdot z\leqslant y\cdot z\rightarrow x\leqslant y$ implies that $\mathcal{K}\models q\leqslant t_1$ (since $\mathcal{K}\models t_2\not=\textsf{0}$). Now, the desired conclusion follows from the induction hypothesis and  $\mathcal{K}\models  b=r+q\cdot t_2\wedge r\leqslant t_2\wedge q\leqslant t_1$.
\end{proof}
\begin{lemma}[Preservation of Atomic Formulas]\label{lem-t1atom}
With the assumptions of Theorem~\ref{th-t1} for any atomic formula $\varphi(x_1,\ldots,x_n)$ and any $i_1,\ldots,i_n\leqslant a$,  we have that

$\mathcal{M}\models\varphi(i_1,\ldots,i_n) \iff
\mathfrak{M}(\Lambda,p)\models
\varphi(\underline{i_1}/p,\ldots,\underline{i_n}/p)$.
\end{lemma}
\begin{proof}
By the $T_1-$axioms $x\not=y\leftrightarrow\textsf{S}(x)\leqslant y\vee\textsf{S}(y)\leqslant x$ and  $x\not\leqslant y\leftrightarrow\textsf{S}(y)\leqslant x$ it suffices to prove the one direction only:  $\mathcal{M}\models\varphi(i_1,\ldots,i_n) \Longrightarrow
\mathfrak{M}(\Lambda,p)\models
\varphi(\underline{i_1}/p,\ldots,\underline{i_n}/p)$. If $\varphi= ``t\leqslant s"$ for some $\mathcal{L}_A-$terms $t$ and $s$, then $\mathcal{M}\models t\leqslant s$ implies the existence of some $b\!\in\!\mathcal{M}$ such that $\mathcal{M}\models b+t=s$.  By the $T_1-$axiom $x\leqslant x+z$, $\mathcal{M}\models b\leqslant s$ so by Lemma~\ref{termt1} there exists an $\mathcal{L}_A-$term $r$ (and some $j_1,\ldots,j_m\leqslant a$) such that $\mathcal{K}\models b=r$. Whence, $\mathcal{M}\models r+t=s$. So, noting that $\mathcal{M},\mathfrak{M}(\Lambda,p)\models T_1$, it suffices to prove the lemma for the atomic formula $\varphi$ of the form $\varphi = ``t=s"$.

For that we first note that if $i_1,\ldots,i_n\leqslant a$ then $t(i_1,\ldots,i_n),s(i_1,\ldots,i_n)\leqslant\omega_1(a)$ holds. Suppose we have  $\mathcal{M}\models t(i_1,\ldots,i_n)=s(i_1,\ldots,i_n)=i$. We show by induction on (the complexity of) $t$ that the condition $\mathcal{M}\models t(i_1,\ldots,i_n)=i$ implies $\mathfrak{M}(\Lambda,p)\models t(\underline{i_1}/p,\ldots,\underline{i_n}/p)=\underline{i}/p$. Let us
 note that the statement $\mathfrak{M}(\Lambda,p)\models t(\underline{i_1}/p,\ldots,\underline{i_n}/p)=\underline{i}/p$ is equivalent to  $\mathcal{M}\models ``p\models t(\underline{i_1},\ldots,\underline{i_n})=\underline{i}"$. So, it suffices to show the equivalence $\mathcal{M}\models t(i_1,\ldots,i_n)=i\leftrightarrow ``p\models t(\underline{i_1},\ldots,\underline{i_n})=\underline{i}"$ by induction on $t$. For $t=\textsf{0}$ and $t=\textsf{S}(t_1)$ the result follows from the definition $\underline{0}=\textsf{0}$ and $\underline{j+1}=\textsf{S}(\underline{j})$. And for $t=t_1+t_2$ and $t=t_1\cdot t_2$ the result follows from the $T_1-$axioms  $x+\textsf{0}=x$,  $x+\textsf{S}(y)=\textsf{S}(x+y)$,
$x\cdot\textsf{0}=\textsf{0}$, and $ x\cdot\textsf{S}(y)=x\cdot y+ x$.
\end{proof}

Hence, the lemma also holds for open formulas $\varphi$ as well. For bounded formulas we note that the range of quantifiers of $\varphi(i_1,\ldots,i_n)$ for $i_1,\ldots,i_n\leqslant a$ is contained in the set $\{j\mid j\leqslant\omega_1(a)\}$. This is formally expressed in the following lemma.
\begin{lemma}[End-Extension Property]\label{endext}
With the assumptions of Theorem~\ref{th-t1}, if for some $i\leqslant a$ and some term $t$ we have $(\mathcal{M}\models)p\models t\leqslant \underline{i}$  then there exists some $j\leqslant i$ such that $(\mathcal{M}\models)p\models t=\underline{j}$.
\end{lemma}
\begin{proof} By induction on the term $\underline{i}$.  For $i=\textsf{0}$, if $p\models t\leqslant\textsf{0}$ then by Lemma~\ref{lem-univ}, and the $T_1-$axiom $x\leqslant\textsf{0}\leftrightarrow x=\textsf{0}$, we have $p\models t=\textsf{0}=\underline{0}$. For $\underline{i}=\textsf{S}({\underline{j}})$, if $p\models t\leqslant\textsf{S}(\underline{j})$ then by Lemma~\ref{lem-univ}, and the $T_1-$axiom $x\leqslant\textsf{S}(y)\leftrightarrow x=\textsf{S}(y)\vee
x\leqslant y$, we must have that $p\models t=\textsf{S}(\underline{j})\vee t\leqslant \underline{j}$. Now the conclusion follows from the induction hypothesis.
\end{proof}

Now we can prove Theorem~\ref{th-t1}.

\begin{proof} ({\bf of Theorem~\ref{th-t1}})
By induction on (the complexity of the bounded formula) $\varphi$. As the lemma has been proved for open formulas $\varphi$, it suffices to show that if the lemma holds for the (bounded) formula $\varphi$ then it also holds for the (bounded) formula $\exists x\leqslant t(i_1,\ldots,i_n)\varphi(x,i_1,\ldots,i_n)$ where $t$ is an $\mathcal{L}_A-$term; in the other words:
\newline\centerline{
$\mathcal{M}\models\exists x\leqslant t(i_1,\ldots,i_n)\varphi(x,i_1,\ldots,i_n) \iff \mathfrak{M}(\Lambda,p)\models\exists x\leqslant t(\underline{i_1}/p,\ldots,\underline{i_n}/p)
\varphi(\underline{i_1}/p,\ldots,\underline{i_n}/p)$.}
\noindent $\bullet$\ If $\mathcal{M}\models b\leqslant t(i_1,\ldots,i_n)\wedge\varphi(b,i_1,\ldots,i_n)$, for some $b\in\mathcal{M}$, then by Lemma~\ref{termt1} there are terms $s$ and elements $j_1,\ldots,j_m\leqslant a$ such that $\mathcal{M}\models b=s(j_1,\ldots,j_m)$. So, we have $\mathcal{M}\models\varphi(s(j_1,\ldots,j_m),i_1,\ldots,i_n)$. Whence, by the induction hypothesis we also have $\mathfrak{M}(\Lambda,p)\models
\varphi(s(\underline{j_1}/p,\ldots,\underline{j_m}/p),\underline{i_1}/p,\ldots,\underline{i_n}/p)$, thus, noting that we already have $\mathfrak{M}(\Lambda,p)\models s(\underline{j_1}/p,\ldots,\underline{j_m}/p)\leqslant t(\underline{i_1}/p,\ldots,\underline{i_n}/p)$, the desired conclusion holds: $\mathfrak{M}(\Lambda,p)\models\exists x\leqslant t(\underline{i_1}/p,\ldots,\underline{i_n}/p)
\varphi(\underline{i_1}/p,\ldots,\underline{i_n}/p)$.

\noindent $\bullet$\ Conversely, if $\mathfrak{M}(\Lambda,p)\models d\leqslant t(\underline{i_1}/p,\ldots,\underline{i_n}/p)\wedge
\varphi(d,\underline{i_1}/p,\ldots,\underline{i_n}/p)$ holds for some $d\in\mathfrak{M}(\Lambda,p)$ then by Lemma~\ref{termt1} there are some $\mathcal{L}_A-$term $s$ and some $l_1,\ldots,l_m\leqslant \underline{a}/p$ such that $\mathfrak{M}(\Lambda,p)\models d=s(l_1,\ldots,l_m)$. For each $\alpha\leqslant m$ there is some term $\ell_\alpha\in\Lambda^{\langle\infty\rangle}$ such that $l_\alpha=\ell_\alpha/p$.  For each such $\alpha$ we also have that  $\mathfrak{M}(\Lambda,p)\models \ell_\alpha/p\leqslant \underline{a}/p$ or equivalently $\mathcal{M}\models ``p\models \ell_\alpha\leqslant \underline{a}"$. So, by Lemma~\ref{endext} there exists some $j_\alpha\leqslant a$ for which we have  $\mathcal{M}\models \ell_\alpha=\underline{j_\alpha}$. Whence, $\mathfrak{M}(\Lambda,p)\models d=s(\underline{j_1}/p,\ldots,\underline{j_m}/p)$ and so

$\mathfrak{M}(\Lambda,p)\models
s(\underline{j_1}/p,\ldots,\underline{j_m}/p)\leqslant t(\underline{i_1}/p,\ldots,\underline{i_n}/p)$, and

$\mathfrak{M}(\Lambda,p)\models
\varphi(s(\underline{j_1}/p,\ldots,\underline{j_m}/p),\underline{i_1}/p,\ldots,\underline{i_n}/p)$.

\noindent Thus, by the induction hypothesis we have

$\mathcal{M}\models
s(j_1,\ldots,j_m)\leqslant t(i_1,\ldots,i_n)$, and
 $\mathcal{M}\models
\varphi(s(j_1,\ldots,j_m),i_1,\ldots,i_n)$.

\noindent So, we conclude that $\mathcal{M}\models\exists x\leqslant t(i_1,\ldots,i_n)\varphi(x,i_1,\ldots,i_n)$.
\end{proof}

Let us repeat where we are now: in looking for a finite fragment $T\subseteq{\rm I\Delta_0}$ such that ${\rm I\Delta_0}\not\vdash{\rm HCon}(T)$ we found a finite fragment $T_0\subseteq{\rm I\Delta_0}$ and a bounded formula $\theta(x)$ such that  $T_0\vdash\neg\exists x\!\in\!\log\mathcal{I}\theta(x)$ but the theory ${\rm I\Delta_0}+\exists x\!\in\!\mathcal{I}\theta(x)$ is consistent and has a model $\mathcal{M}\models{\rm I\Delta_0}+[a\!\in\!\mathcal{I}\wedge\theta(a)]$. Then we aim at showing that $\mathcal{M}\not\models{\rm HCon}(T)$. If $\mathcal{M}\models{\rm HCon}(T)$ then we form the set of formulas $\Gamma=\{\underline{i}\mid i\leqslant\omega_1(a)\}$ for which $\omega_2(\ulcorner\Gamma\urcorner)$ exists (by the very definition of $\mathcal{I}$ and the assumption $a\!\in\!\mathcal{I}$), and so we can form the model $\mathfrak{M}(\Gamma,p)$ where $p$ is an $T-$evaluaiton on $\Gamma^{\langle j\rangle}$ (where $j\leqslant\log^4(\ulcorner\Gamma\urcorner)$ can be taken to be non-standard if $a$ is so). The theory $T_1$ had the property that $\mathfrak{M}(\Gamma,p)\models\theta(\underline{a}/p)$ (by Theorem~\ref{th-t1}), and in the next subsection we introduce a finite fragment $T_2\subseteq{\rm I\Delta_0}$ such that for a suitable $\Lambda\supseteq\Gamma$ (to be defined later) we will have $\mathfrak{M}(\Lambda,p)\models \underline{a}/p\!\in\!\log\mathcal{I}$. Then by taking $T$ to be any finite fragment of ${\rm I\Delta_0}$ which extends $T_0\cup T_1\cup T_2$ we will conclude that $\mathcal{M}\models\neg{\rm HCon}(T)$.
\subsection{The Third Finite Fragment}
The fragments $T_0$ and $T_1$ were chosen not by their axioms but by their implications; $T_0$ had to prove $\neg\exists x\in\log\mathcal{I}\theta(x)$ (Definition~\ref{def-t0}), and $T_1$ had to prove some certain arithmetical statements (Definition~\ref{def-t1}). But for $T_2$ we require that it contains one of the following sentences as (one of) its (explicit) axioms (not only its consequences).
\begin{definition}[Axioms for Totality of Squaring Function]\label{def-sq}

(1) The induction principle for the bounded formula $\psi(x) = ``\exists y\leqslant x^2[y=x\cdot x]"$ is denoted by ${\rm Ind}_{\square}: \ \psi(\textsf{0})\wedge\forall x\big(\psi(x)\rightarrow\psi(\textsf{S}(x))\big)\rightarrow\forall x\psi(x)$. Or, in other words (cf. Examples~\ref{q1example},\ref{q2example}) ${\rm Ind}_{\square}$, which is an axiom of the theory ${\rm I\Delta_0}$, is the sentence:
\newline\centerline{
    $\exists y\leqslant \textsf{0}^2[y=\textsf{0}\cdot \textsf{0}]\wedge\forall x\big(\exists y\leqslant x^2[y=x\cdot x]\rightarrow\exists y\leqslant \textsf{S}(x)^2[y=\textsf{S}(x)\cdot \textsf{S}(x)]\big)\Longrightarrow\forall x\exists y\leqslant x^2[y=x\cdot x]$.}

(2) The $\Pi_1-$sentence expressing the totality of  squaring  is denoted by  $\Omega_0: \ \forall x\exists y\leqslant x^2[y=x\cdot x]$.    \hfill $\lozenge\!\!\!\!\!\lozenge$
\end{definition}
We denote by $\textswab{q}(x)$ the Skolem function symbol of the formula $\exists y\leqslant x^2[y=x\cdot x]$  (cf. Examples~\ref{q1example},\ref{q2example}). Then the Skolemized forms of the axioms of Definition~\ref{def-sq} will be as
\begin{enumerate}
\item $[u\not\leqslant\textsf{0}^2\vee u\not=\textsf{0}\cdot\textsf{0}]\ \bigvee$

$\big[[\textswab{q}(\textswab{c})\leqslant\textswab{c}^2
  \wedge\textswab{q}(\textswab{c})=\textswab{c}\cdot\textswab{c}]
  \wedge [v\not\leqslant\textsf{S}(\textswab{c})^2\vee v\not=\textsf{S}(\textswab{c})\cdot\textsf{S}(\textswab{c})]\big]\ \bigvee$

$[\textswab{q}(x)\leqslant x^2\wedge\textswab{q}(x)=x\cdot x]$,

where $u,v,x$ are free variables and $\textswab{c}$ is the Skolem constant as in Example~\ref{q1example}.
\item $\textswab{q}(x)\leqslant x^2 \wedge \textswab{q}(x)=x\cdot x$.
\end{enumerate}
Define the terms $\textsf{q}_i$'s by induction: $\textsf{q}_0=\textsf{S}(\textsf{S}(\textsf{0}))$ and $\textsf{q}_{i+1}=\textswab{q}(\textsf{q}_i)$. It can be easily seen that $\textsf{q}_i$ represents the number $\exp^{2}(i)$, while for the code of $\textsf{q}_i$ we have $\log(\ulcorner\textsf{q}_i\urcorner)\in\mathcal{O}\big(\log(\exp(i))\big)$. That is to say that while the value of the term $\textsf{q}_i$ is of double exponential, the code of it is of (single) exponential. This (one) exponential gap, will make our proof to go through.

Formulating the statement $``x\in\textfrak{log}^2"$ can be stated as ``there  exists a sequence $s$ such that $(s)_0=2$ and $|s|=x+1$ and for any $i<x$ we have $(s)_{i+1}=(s)_i\cdot(s)_i$". And $``y\in\log\mathcal{I}"$ can be stated as $``4y^4\in\textfrak{log}^2"$. Put $\Upsilon=\{\textsf{q}_i\mid i\leqslant 4a^4\}$. Then any $\Omega_0-$evaluaton or ${\rm Ind}_\square-$evaluation on $\Upsilon^{\langle\infty\rangle}$ must satisfy $\textsf{q}_{i+1}=\textsf{q}_i\cdot\textsf{q}_i$ for any $i<4a^4$.  If $p$ is any such evaluation, then $\mathfrak{M}(\Upsilon,p)\models \forall i<4(\underline{a}/p)^4[\textsf{q}_{i+1}/p=\textsf{q}_i/p\cdot\textsf{q}_i/p]$. We require the finite fragment $T_2\subseteq{\rm I\Delta_0}$ to have the property that for any model $\mathcal{K}\models T_2$ if there are elements $q_0,q_1,\ldots,q_b\in\mathcal{K}$ such that $\mathcal{K}$ satisfies $q_0=2$ and $q_{i+1}=q_i^2$ for any $i<b$, then $\mathcal{K}\models b\in\textfrak{log}^2$.  Let us note that the code of the sequence $\langle\exp^2(0),\exp^2(1),\ldots,\exp^2(b)\rangle$ is roughly bounded by $\prod_{i\leqslant b}\exp^2(i) \approx (\exp^2(b))^2=\exp^2(b+1)$. So, in the presence of $q_0,q_1,\ldots,q_b\in\mathcal{K}$ with the above property, the (code of the) sequence $s$ with the property ``$(s)_0=2$,  $|s|=x+1$ and for any $i<x$, $(s)_{i+1}=(s)_i\cdot(s)_i$" must exist. Note also that ${\rm I\Delta_0}\vdash\forall i[i\in\textfrak{log}^2\rightarrow i+1\in\textfrak{log}^2]$. \qquad\qquad $(\divideontimes)$
\begin{definition}[The Third Fragment $T_2$]\label{def-t2}

(1) If the usual axiomatization of ${\rm I\Delta_0}$ is taken into account, then let $T_2$ be a finite fragment of it which contains the axiom ${\rm Ind}_\square$ and has the property $(\divideontimes)$ above.

(2) If ${\rm I\Delta_0}$ has been axiomatized all by $\Pi_1-$formulas, where the induction axioms are in the form
    \newline\centerline{$\forall y\big(\varphi(\textsf{0})\wedge\forall x<y[\varphi(x)\rightarrow\varphi(\textsf{S}(x))]\rightarrow\forall x\leqslant y\varphi(x)\big)$} for bounded $\varphi$, then we take the theory $T_2$ to be a finite fragment of ${\rm I\Delta_0^\Pi}+\Omega_0$, where ${\rm I\Delta_0^\Pi}$ is the above $\Pi_1-$axiomaitzation of ${\rm I\Delta_0}$, together with the axiom $\Omega_0$, such that it has the property $(\divideontimes)$ above. So, in this case $T_2$ is a $\Pi_1-$theory.\hfill $\lozenge\!\!\!\!\!\lozenge$
\end{definition}
Let us reiterate the main property of $T_2$ again.
\paragraph{\bf The Main Property of $T_2$} For a model $\mathcal{K}\models T_2$ if there are $q_0,q_1,\ldots,q_b\in\mathcal{K}$ such that for any $j<b$ we have $\mathcal{K}\models q_{j+1}=q_j^2$ then $\mathcal{K}\models``b\in\textfrak{log}^2"$. \hfill$\lozenge\!\!\!\!\!\lozenge$
\subsection{The Proof of the Main Result}
Let $T$ be any finite fragment of ${\rm I\Delta_0}$ or ${\rm I\Delta_0^\Pi}+\Omega_0$ such that $T\supseteq T_0\cup T_1\cup T_2$. If $T_2$ is taken as in the clause (1) of Definition~\ref{def-t2} then $T$ is truly a finite fragment of ${\rm I\Delta_0}$, and if $T_2$ is taken as in the clause (2) of Definition~\ref{def-t2} then $T$ is a finite ${\rm I\Delta_0}-$derivable $\Pi_1-$theory, whose conjunction (denoted by $U$) is a ${\rm I\Delta_0}-$derivable $\Pi_1-$sentence.
\begin{theorem}[The Main Theorem]\label{th-main}

\noindent (1) For a finite fragment $T$ of ${\rm I\Delta_0}$ we have ${\rm I\Delta_0}\not\vdash{\rm HCon}(T)$.

\noindent (2) There exists an ${\rm I\Delta_0}-$derivable $\Pi_1-$sentence $U$ such that  ${\rm I\Delta_0}\not\vdash{\rm HCon}(U)$.
\end{theorem}
\begin{proof}
For the part (1) take $T_2$ as in clause~(1) of Definition~\ref{def-t2}, and for part (2) take $T_2$ as in clause~(2) of Definition~\ref{def-t2}, and let $U$ be the conjunction of the axioms of $T$. In each case we will have the Skolem function symbol $\textswab{q}(x)$ for squaring $x\mapsto x^2$.

By Theorem~\ref{th-ilogi} there exists a (fixed) bounded formula $\theta(x)$, for the cut $\mathcal{I}$ defined in Definition~\ref{def-cuti}, such that ${\rm I\Delta_0}\not\vdash\neg\exists x\in\!\mathcal{I}\!\theta(x)$ and $T_0\vdash\neg\exists x\!\in\!\log\mathcal{I}\theta(x)$ (see Definition~\ref{def-t0}). Fix $\mathcal{M}\models {\rm I\Delta_0}+[a\!\in\!\mathcal{I}\wedge\theta(a)]$. We show that $\mathcal{M}\not\models{\rm HCon}(T)$.

Assume, for the sake of contradiction, that $\mathcal{M}\models{\rm HCon}(T)$. Define the terms $\underline{i}$'s and $\textsf{q}_i$'s by induction: $\underline{0}=\textsf{0}$, $\underline{i+1}=\textsf{S}(\underline{i})$, $\textsf{q}_0=\underline{2}$, $\textsf{q}_{i+1}=\textswab{q}(\textsf{q}_i)$. Let $\Lambda$ be the set of terms $\{\underline{i}\mid i\leqslant\omega_a(a)\}\cup\{\textsf{q}_i\mid i\leqslant\omega_1(a)\}$ in $\mathcal{M}$.  As we saw earlier, the code of $\underline{i}$ (and $\textsf{q}_i$) are  bounded by some polynomial of $\exp(i)$ and the code of the $\Lambda$ is polynomially bounded by $\exp\big((\omega_1(a)^2)\big)$ or $\exp^2\big(2(\log a)^2\big)$, and finally $\omega_2(\ulcorner\Lambda\urcorner)$ is polynomially bounded by $\exp^2\big(4(\log a)^4\big)$; which exists by the assumption $a\!\in\!\mathcal{I}$. We note that $a$ is non-standard, because otherwise we would have $a\!\in\!\log\mathcal{I}$ and whence $\mathcal{M}$ would be a model of the inconsistent theory ${\rm I\Delta_0}+\exists x\!\in\!\log\mathcal{I}\theta(x)$; a contradiction with the hypothesis. The existence of $\omega_2(\ulcorner\Lambda\urcorner)$ assures the existence of a non-standard element $j (\leqslant\log^4(\ulcorner\Lambda\urcorner))$ for which $\Lambda^{\langle j\rangle}$ exists, and so by the assumption $\mathcal{M}\models{\rm HCon}(T)$ there must exist some $T-$evaluation $p$ on $\Lambda^{\langle j\rangle}$ (hence, on $\Lambda^{\langle\infty\rangle}$) in $\mathcal{M}$.  So, we can form the model $\mathfrak{M}(\Lambda,p)$. For this model we have $\mathfrak{M}(\Lambda,p)\models T$ by Lemma~\ref{lem-hmodel}. Since $\mathcal{M}\models\theta(a)$ (and $\mathfrak{M}(\Lambda,p)\models T_1$) then $\mathfrak{M}(\Lambda,p)\models\theta(\underline{a}/p)$ by Theorem~\ref{th-t1}. Also, since $\mathfrak{M}(\Lambda,p)\models T_2$ and $\textsf{q}_0,\textsf{q}_1,\ldots,\textsf{q}_b$ (for $b=4a^4$) are elements of $\mathfrak{M}(\Lambda,p)$ such that $\mathfrak{M}(\Lambda,p)\models \textsf{q}_0=2$ and $\mathfrak{M}(\Lambda,p)\models \textsf{q}_{i+1}=\textsf{q}_i^2$ for any $i<b$, then (by the main property of $T_2$) $\mathfrak{M}(\Lambda,p)\models``b\in\textfrak{log}^2"$. Or, in other words, $\mathfrak{M}(\Lambda,p)\models``\underline{a}/p\!\in\!\log\mathcal{I}"$. Whence, $\mathfrak{M}(\Lambda,p)\models [\underline{a}/p\!\in\!\log\mathcal{I}\wedge\theta(\underline{a}/p)]$.
So, $\mathfrak{M}(\Lambda,p)$ is a model of $T+\exists x\!\in\!\log\mathcal{I}\theta(x)$, and this is contradiction with the assumption of $T\supseteq T_0$ and the inconsistency of the theory $T_0+\exists x\!\in\!\log\mathcal{I}\theta(x)$. Thus $\mathcal{M}\not\models{\rm HCon}(T)$ and so ${\rm I\Delta_0}\not\vdash{\rm HCon}(T)$.
\end{proof}


\end{document}